\documentclass{amsart}
\usepackage{graphicx}
\usepackage{latexsym}
\usepackage{hyperref}
\usepackage{amsfonts}
\usepackage[all]{xy}
\usepackage{amssymb, mathrsfs, amsfonts, amsmath}
\usepackage{amsbsy}
\usepackage{amsfonts}
\usepackage{setspace}
\newtheorem*{acknowledgement}{Acknowledgement}
\newtheorem{corollary}{Corollary}
\newtheorem{definition}{Definition}
\newtheorem{lemma}{Lemma}
\newtheorem{proposition}{Proposition}

\newtheorem{theorem}{Theorem}
\newtheorem{example}{Example}

\numberwithin{equation}{section}

\renewcommand{\div}{{\rm div}}
\newcommand{\R}{\mathring{R}}
\newcommand{\Rc}{\mathring{Ric}}
\begin{document}

\title[quasi Einstein manifolds]{Compact quasi-Einstein manifolds with boundary}
\author{Rafael Di\'ogenes}
\author{Tiago Gadelha}

\address[R. Di\'ogenes]{UNILAB, Instituto de Ci\^encias Exatas e da Natureza, Rua Jos\'e Franco de Oliveira, s/n, 62790-970, Reden\c{c}\~ao / CE, Brazil.}\email{rafaeldiogenes@unilab.edu.br}

\address[T. Gadelha]{Instituto Federal do Cear\'a - IFCE, Campus Maracana\'u, Av. Parque Central, S/n, 61939-140, Maracana\'u / CE, Brazil.}\email{tiago.gadelha@ifce.edu.br}

\thanks{T. Gadelha was partially supported by FUNCAP/Brazil and CAPES/Brazil - Finance code 001.}

\begin{abstract}
The goal of this article is to study compact quasi-Einstein manifolds with boundary. We provide boundary estimates for compact quasi-Einstein manifolds simi\-lar to previous results obtained for static and $V$-static spaces. In addition, we show that compact quasi-Einstein manifolds with connected boundary and satisfying a suitable pinching condition must be isometric, up to scaling, to the standard hemisphere $\mathbb{S}_{+}^{n}.$ 
\end{abstract}

\date{\today}



\maketitle

\section{Introduction}
A classical topic in Riemannian Geometry is to built explicit examples of Einstein manifolds. Einstein manifolds are not only interesting in themselves but are also related to many important topics of Riemannian geometry. As discussed by Besse \cite[pg. 265]{Besse}, one promi\-sing way to construct Einstein metrics is by imposing symmetry, such as by considering warped products. In this context, $m$-quasi-Einstein manifolds play a crucial role.

According to \cite{CaseShuWey}, a complete Riemannian manifold $(M^n,\,g),$ $n\geq 2,$ will be called $m$-{\it quasi-Einstein manifold}, or simply {\it quasi-Einstein manifold}, if there exists a smooth potential function $f$ on $M^n$ satisfying the following fundamental equation
\begin{equation}
\label{eqqem}
Ric_{f}^{m}=Ric+\nabla ^2f-\frac{1}{m}df\otimes
df=\lambda g,
\end{equation} for some constants $\lambda$ and $m\neq 0,$ where $\nabla ^2f$ stands for the Hessian of $f.$ It is also important to recall that, on a quasi-Einstein manifold, there is a constant  $\mu$ such that 
\begin{equation}\label{2eq}
\Delta_{f} f = m\lambda-m\mu e^{\frac{2}{m}f},
\end{equation} where $\Delta_{f}=\Delta -\langle\nabla f,\,\cdot\,\,\rangle$ is the $f$-Laplacian. For more details, we refer the reader to \cite{KK}. 

The $m$-Bakry-Emery Ricci tensor $Ric_{f}^{m}$ plays a crucial role in the study of Einstein warped products. For instance, when $m$ is a positive integer, an $m$-quasi-Einstein manifold corresponds to a base of Einstein warped products; for more details see Corollary 9.107 in \cite[pg. 267]{Besse} (see also Theorem 1 in \cite{Ernani2}). Another interesting motivation comes from the study of diffusion operators by Bakry and \'Emery.  When $m=1$ we consider in addition that $\Delta u=-\lambda u$ and in this case, such metrics are commonly called \textit{static spaces}. Static spaces have been studied extensively for their connections to scalar curvature, the positive mass theorem, and general relativity (see \cite{AN}). In other words, $m$-quasi-Einstein manifolds can be seen as a generalization of static spaces. Furthermore, $\infty$-quasi-Einstein manifold is a gradient Ricci soliton, i.e, self-similar solutions of the Ricci flow. There is a vast literature on quasi-Einstein manifolds, we refer to the reader, for instance \cite{Ernani2,BRR2019,Besse,CaseShuWey,Case,He-Petersen-Wylie2012,Ernani_Keti,rimoldi,Wang2011}.

Assuming that $m<\infty,$ we may set a function $u = e^{-\frac{f}{m}}$ on $M^n.$ Hence, we immediately get 
 
 \begin{equation}
 \label{1a}
 \nabla u = -\frac{u}{m}\nabla f
 \end{equation} as well as

\begin{equation}\label{fu}
Hess f - \frac{1}{m}df \otimes df = -\frac{m}{u}Hess\, u.
\end{equation} In this situation, we follow the approach used by He, Petersen and Wylie \cite{He-Petersen-Wylie2012} in order to study quasi-Einstein metrics on compact manifolds with boundary. To fix notation, we consider the following definition (see \cite{He-Petersen-Wylie2012}).

\begin{definition}
\label{def}
A complete Riemannian manifold $(M^n,\,g),$ $n\geq 2,$ (possibly with boundary) will be called $m$-{\it quasi-Einstein manifold}, or simply {\it quasi-Einstein manifold}, if there exists a smooth potential function $u$ on $M^n$ satisfying the following fundamental equation
\begin{equation}\label{eqdef}
\nabla^{2}u = \dfrac{u}{m}(Ric-\lambda g),
\end{equation} for some constants $\lambda$ and $m$. Moreover, $u>0 $ in $int(M)$ and $u=0$ on $\partial M$. Here, $\nabla^{2} u$ and $Ric$ stand for the Hessian of $u$ and the Ricci curvature of $M^n,$ respectively.
\end{definition}

In \cite{He-Petersen-Wylie2012} and \cite{Petersen and Chenxu}, He, Petersen and Wylie studied compact quasi-Einstein manifold with no empty boundary. They obtained a classification result for quasi-Einstein manifolds which are also Einstein manifold. Moreover, they provided some nontrivial examples, including the Generalized Schwarzschild metric with $\lambda=0$ (see \cite{He-Petersen-Wylie2012}). Moreover, they studied $m$-quasi-Einstein manifolds with constant scalar curvature (see  \cite{Petersen and Chenxu}). In this article, we focus in nontrivial compact $m$-quasi-Einstein manifolds with boundary. Hence, by Theorem 4.1 of \cite{He-Petersen-Wylie2012}, they have necessarily $\lambda>0.$ Moreover, by Remark 5.1 of \cite{He-Petersen-Wylie2012} the scalar curvature is positive more precisely $R\geq\frac{n(n-1)}{m+n-1}\lambda.$ Among the examples obtained in \cite[Proposition 2.4]{Petersen and Chenxu}, it is very important to highlight an example built on the stardard hemisphere.

\begin{example}
\label{example}
Let  $\Bbb{S}^n_+$ be a standard hemisphere with metric $g=dr^2+\sin^2r g_{\Bbb{S}^{n-1}}$ and potential function $u(r)=\cos r,$ where $r$ is a height function and $r\leq\frac{\pi}{2}.$ We consider $m$ to be an arbitrary (but finite) constant and $\lambda=m+n-1.$ Thus, $\Bbb{S}^n_+$ is a compact, oriented $m$-quasi-Einstein manifold with boundary $\Bbb{S}^{n-1}.$ Indeed, it is easy to check that $u=\cos\frac{\pi}{2}=0$ on the boundary and $u>0$ on interior of $\Bbb{S}^n_+.$ Moreover, since $\nabla^2r=\sin r\cos r g_{\Bbb{S}^{n-1}}$ (see Chapter 3, Section 2 of \cite{petersen}) we obtain $\nabla u=-\sin r\nabla r$ and 
\begin{eqnarray*}
\nabla^2u&=&-\sin r\nabla^2r-\cos rdr\otimes dr\\
 &=&-\sin^2r\cos r g_{\Bbb{S}^{n-1}}-\cos rdr^2\\
 &=&-\cos r\left(dr^2+\sin^2rg_{\Bbb{S}^{n-1}}\right)\\
 &=&-ug.
\end{eqnarray*}

On the other hand, we already know that $Ric=(n-1)g$ and therefore, we obtain 
\begin{eqnarray*}
\frac{u}{m}(Ric-\lambda g)&=&\frac{u}{m}\left((n-1)g-(m+n-1)g\right)\\
 &=&-ug,
\end{eqnarray*} so that $$\nabla^2u=\frac{u}{m}(Ric-\lambda g),$$ as asserted. 
\end{example}

In the sequel, we present an example of compact $m$-quasi-Einstein manifold with disconnected boundary.

\begin{example}\label{example2}
Let $\lambda>0$ be a real constant and $M=[0,\frac{\sqrt{m}}{\sqrt{\lambda}}\pi]\times\Bbb{S}^{n-1}$ be a product Riemannian with metric $g=dt^2+\frac{n-2}{\lambda}g_{\Bbb{S}^{n-1}}$ and potential function $u(r)=\sin\left(\frac{\sqrt{\lambda}}{\sqrt{m}}t\right).$ Thus, $M$ is a compact, oriented $m$-quasi-Einstein manifold with disconnected boundary (the boundary is the union of two copies of $\Bbb{S}^{n-1}).$ To prove this, it suffices to notice that $u>0$ on $int(M)$ and taking into account that $$\partial M=\{0\}\times\Bbb{S}^{n-1}\cup\left\{\frac{\sqrt{m}}{\sqrt{\lambda}}\pi\right\}\times\Bbb{S}^{n-1}=\partial M_1\cup\partial M_2,$$ we have $u=\sin(0)=0$ on $\partial M_1$ and $u=\sin(\pi)=0$ on $\partial M_2.$ Therefore, $u=0$ on $\partial M.$  Moreover, since $\nabla^2t=0$ (see Chapter 3, Section 2 of \cite{petersen}) we get $$\nabla u =\frac{\sqrt{\lambda}}{\sqrt{m}}\cos\left(\frac{\sqrt{\lambda}}{\sqrt{m}}t\right)\nabla t$$ and 
\begin{eqnarray*}
\nabla^2u&=&\frac{\sqrt{\lambda}}{\sqrt{m}}\cos\left(\frac{\sqrt{\lambda}}{\sqrt{m}}t\right)\nabla^2t-\frac{\lambda}{m}\sin\left(\frac{\sqrt{\lambda}}{\sqrt{m}}t\right)dt^2\\
 &=&-\frac{\lambda}{m}udt^2.
\end{eqnarray*} Next,  recalling that $Ric=(n-2)g_{\Bbb{S}^{n-1}}$ we obtain
\begin{eqnarray*}
\frac{u}{m}(Ric-\lambda g)&=&\frac{u}{m}\left((n-2)g_{\Bbb{S}^{n-1}}-\lambda dt^2-(n-2)g_{\Bbb{S}^{n-1}}\right)\\
 &=&-\frac{\lambda}{m}udt^2.
\end{eqnarray*} So, it follows that $$\nabla^2u=\frac{u}{m}(Ric-\lambda g),$$ as we wanted to prove. 
\end{example}

It should be remarked that if a compact nontrivial $m$-quasi-Einstein manifold with no empty boundary has constant scalar curvature, then $R\leq \lambda n.$ Indeed, taking the trace in (\ref{eqdef}) we deduce $m\Delta u=u(R-\lambda n).$ Hence, if $R>\lambda n,$ then $\Delta u>0.$ But, by using the Maximum Principle we conclude that $u$ have the maximum on the boundary and therefore, $u$ is identically zero, which leads to a contradiction. Note that in Example \ref{example} we have $R=n(n-1)$ and $\lambda=m+n-1,$ thus $R-n\lambda=-mn.$ In Example \ref{example2} we have $R=(n-1)\lambda$ and $R-n\lambda=-\lambda,$ for any $\lambda>0.$

It is natural to ask whether Example \ref{example}  is the only simply connected and compact $m$-quasi-Einstein manifold with one connected boundary component. This question can be investigate by several ways. In this paper we will treat this problem. Before to state our main results, let us  provide a historical motivation comes from the study of static and $V$-static spaces. A result by Shen \cite{Shen} and Boucher-Gibbons-Horowitz \cite{BGH} asserts that the boundary $\partial M$ of a compact $3$-dimensional oriented static manifold with connected boundary and scalar curvature 6 must be a 2-sphere whose area satisfies $|\partial M|\leq 4\pi.$ Moreover, the equality holds if and only if $M^{3}$ is equivalent to the standard hemisphere. Batista, Di\'{o}genes, Ranieri and Ribeiro \cite{Batista-Diogenes-Ranieri-Ribeiro JR} showed the boundary of a $3$-dimensional compact, oriented critical metric of the volume functional with connected boundary $\partial M$ and nonnegative scalar curvature, must be a 2-sphere and its area satisfies $|\partial M|\leq \frac{4\pi}{C(R)}$, where $C(R)$ is a positive constant. Moreover, the equality holds if and only if $(M^{3},g)$ is isometric to a geodesic ball in a simply connected space form $\mathbb{R}^{3}$ or  $ \mathbb{S}^{3}.$ In the same spirit, Barros and Silva \cite{BS} proved that a compact, oriented static triple with connected boundary $\partial M$ and constant positive normalized scalar curvature $R=n(n-1),$ then $|\partial M|\leq\omega_{n-1},$ where $\omega_{n-1}$ is the volume of unitary sphere of dimension $n-1.$ Moreover, the equality is attained only for a round hemisphere $S_{+}^{n}.$ This result provides a partial answer to the cosmic no-hair conjecture formulated by Boucher {\it et al.} \cite{BGH}. Similar boundary estimates were obtained recently by Coutinho {\it et al.} \cite{CDLR} for static perfect fluid space-time of arbitrary dimension.

Based on the above result, Example \ref{example} and taking into account that $m$-quasi-Einstein manifolds can be seen as a generalization of static spaces (i.e., $m=1$), in this article, mainly inspired on ideas developed in \cite{Lucas Ambrosio,Batista-Diogenes-Ranieri-Ribeiro JR,BS,CDLR}, we will investigate boundary estimate for compact $m$-quasi-Einstein manifolds with connected boundary. In this sense, we have established the following.

\begin{theorem}\label{th1} Let $\big(M^{n},\,g,\,u,\,\lambda \big),$ $n\geq 4,$ be a compact, oriented $m$-quasi-Einstein manifold with connected boundary. Suppose that $Ric^{\partial M}\geq \dfrac{R^{\partial M}}{n-1}g_{_{\partial M}}$ with $\inf _{\partial M}R^{\partial M}>0$ and that the scalar curvature of $M^n$ satisfies $R\leq n\lambda.$ Then it holds
\begin{equation*}
|\partial M| \leq {\omega_{n-1}}\left(\frac{n(n-1)}{R_{min}}\right)^{\frac{n-1}{2}}.
\end{equation*} Moreover, the equality holds if and only if $(M^{n},\,g)$ is isometric, up to scaling, to the hemisphere $(\Bbb{S}^n_+,\,g)$ given by Example \ref{example}.	
\end{theorem}

It should be emphasized that Theorem \ref{th1} can be seen as a generalization of Theorem 4 in \cite{BS} for static spaces, i.e., $1$-quasi-Einstein manfold. Next, in the $3$-dimensional case, we have the following result.

\begin{theorem}\label{th2}
Let $\big(M^{3},\,g,\,u,\,\lambda \big)$ be a $3$-dimensional compact, oriented $m$-quasi-Einstein manifold with connected boundary. Suppose that the scalar curvature of $M^3$ satisfies $R\leq3\lambda.$ Then, it holds
\begin{equation*}
|\partial M| \leq\frac{24\pi}{R_{min}}.
\end{equation*} Moreover, the equality holds if and only if $(M^{3},g)$ is isometric, up to scaling, to the hemisphere $(\Bbb{S}^3_+,g)$ given by Example \ref{example}.
\end{theorem}

Notice that Theorem \ref{th2} implies that the area of the boundary of Example \ref{example} is the maximum possible among all compact $m$-quasi-Einstein manifold with connected boundary and satisfying $R\leq3\lambda.$

Similar to results obtained by Batista {\it et al.} \cite{Batista-Diogenes-Ranieri-Ribeiro JR}, in the second part of this article, we will provide a B\"ochner type formula for $m$-quasi-Einstein manifolds. To be precise, Batista, Di\'{o}genes, Ranieri and Ribeiro \cite{Batista-Diogenes-Ranieri-Ribeiro JR}, mainly motivated by \cite{Lucas Ambrosio}, obtained a B\"ochner type formula for critical metrics of the volume functional. Moreover, they used such a formula to obtain rigidity results for such critical metrics. In \cite{Baltazar-Ribeiro JR}, Baltazar and Ribeiro extended this result for higher dimension and moreover, they provided rigidity results for $V-$static spaces in a more general setting. Recently, Coutinho, Di\'ogenes, Leandro and Ribeiro \cite{CDLR} determined a B\"ocher type formula for Riemannian manifolds that satisfies the condition $f\Rc= \mathring{\nabla^2}f,$ which includes a large class of spaces. 

In the sequel, we going to provide a B\"ochner type formula for $m$-quasi-Einstein manifolds.

\begin{theorem}\label{th3}
Let $\big(M^{n},\,g,\,u,\,\lambda \big)$ be a $m$-quasi-Einstein manifold. Then we have
\begin{eqnarray*}
\frac{1}{2}\div \big(u\nabla|\mathring{Ric}|^{2}\big)&=&u|\nabla\mathring{Ric}|^{2}+\frac{m(n-2)u}{m+n-2}|C|^2+u\langle\nabla^2R,\Rc\rangle+\frac{m+2n-2}{n-1}\Rc(\nabla u,\nabla R)\nonumber\\
	&&+\frac{m-1}{2}\langle\nabla|\Rc|^2,\nabla u\rangle+\frac{2Ru}{n-1}|\Rc|^2+\frac{2nu}{n-2}{\rm trace}(\Rc^3)\nonumber\\
	&&-2uW_{ikjp}\R_{ij}\R_{kp}-\frac{m^2(n-2)}{m+n-2}C_{ijk}W_{ijkl}\nabla_lu\nonumber,
\end{eqnarray*} where ${\rm trace}(\Rc^3)=\R_{ij}\R_{jk}\R_{ki}$ and $\mathring{Ric}=Ric-\frac{R}{n}g.$
\end{theorem}

We recall that a Riemannian manifold  $(M^{n},\,g)$ has zero radial Weyl curvature if $W(\cdot,\cdot,\cdot,\nabla u)=0.$ Thus, as a consequence of Theorem \ref{th3} we obtain the following corollary.

\begin{corollary}\label{corWeylZero}
Let $\big(M^{n},\,g,\,u,\,\lambda\big),$ $n\geq 4,$ be a compact $m$-quasi-Einstein manifold with boun\-dary, constant scalar curvature and zero radial Weyl curvature. If $$|\Rc|^2<\frac{R^2}{n(n-1)},$$ then $(M^{n},g)$ is isometric, up to scaling, to the hemisphere $(\Bbb{S}^n_+,g)$ given by Example \ref{example}.
\end{corollary}

Finally, taking into account that the Weyl tensor vanishes in dimension three, we immediately deduce the following result. 

\begin{corollary}\label{cordim3}
Let $\big(M^{3},\,g,\,u,\,\lambda \big)$ be a compact $m$-quasi-Einstein manifold with boundary and constant scalar curvature. Suppose that $$|\Rc|^2<\frac{R^2}{6}.$$ Then $(M^{3},g)$ is isometric, up to scaling, to the hemisphere $(\Bbb{S}^3_+,g)$ given by Example \ref{example}.
\end{corollary}

\section{Preliminaries}
\label{preliminares}

In this section, we will present some key lemmas that will be useful for the establishment of the desired results. Firstly, we recall that the fundamental equation of a $m$-quasi-Einstein manifold $(M^{n},g,u)$ (possibly with boundary) is given by 
\begin{equation}\label{fundamental equation}
\nabla ^{2}u = \dfrac{u}{m}(Ric-\lambda g),
\end{equation} where $u>0$ in interior of $M$ and $u=0$ on $\partial M.$ In particular, taking the trace of (\ref{fundamental equation}) we arrive at
\begin{equation}\label{laplaciano}
\Delta u = \frac{u}{m}(R-\lambda n).
\end{equation} This jointly with (\ref{fundamental equation}) gives
\begin{equation}\label{traceless Ricci}
u\, \mathring{Ric} = m \mathring{\nabla^2}\, u,
\end{equation} where $\mathring{T}=T-\dfrac{{\rm tr}T}{n}g$ stands the traceless part of $T.$

It is important to remember that the Weyl tensor $W$ is defined by the following decomposition formula
\begin{eqnarray}\label{weyl tensor}
R_{ijkl}&=&W_{ijkl}+\dfrac{1}{n-2}\big (R_{ik}g_{jl}+R_{jl}g_{ik}-R_{il}g_{jk}-R_{jk}g_{il}\big) \nonumber \\ 
&&-\dfrac{R}{(n-1)(n-2)}\big (g_{jl}g_{ik}-g_{il}g_{jk}\big ),
\end{eqnarray} where $R_{ijkl}$ denotes the Riemann curvature tensor whereas the Cotton tensor is defined by
\begin{eqnarray}\label{cotton tensor}
C_{ijk}=\nabla_{i}R_{jk}-\nabla_{j}R_{ik}-\dfrac{1}{2(n-1)}\big(\nabla_{i}Rg_{jk}-\nabla_{j}Rg_{ik}\big).
\end{eqnarray} An important relation between Weyl and Cotton tensors is given by
\begin{equation}\label{cottonwyel}
\displaystyle{C_{ijk}=-\frac{(n-2)}{(n-3)}\nabla_{l}W_{ijkl},}
\end{equation} for $n\ge 4.$ Moreover, it is easy to see that
\begin{equation}\label{simetC}
C_{ijk}=-C_{jik} \,\,\,\,\,\,\hbox{and}\,\,\,\,\,\,C_{ijk}+C_{jki}+C_{kij}=0.
\end{equation} In particular, we have
\begin{equation}\label{freetraceC}
g^{ij}C_{ijk}=g^{ik}C_{ijk}=g^{jk}C_{ijk}=0.
\end{equation}

Now we going to present some lemmas that play crucial role in this paper. In our first lemma we follow the ideas of Lemma 1 of \cite{Bach-Flat} for critical metrics of the volume functional. This lemma was also proved by He, Petersen and Wylie with a different notation (see Proposition 6.2 of \cite{He-Petersen-Wylie2012}).

\begin{lemma}\label{lemmaDifference}
Let $\big(M^{n},g,u,\lambda\big)$ be a $m$-quasi-Einstein manifold. Then we have
\begin{eqnarray*}
u\big ( \nabla_{i} R_{jk}-\nabla_{j} R_{ik}\big)=mR_{ijkl}\nabla_{l}u+\lambda \big (\nabla _{i}ug_{jk}-\nabla_{j}ug_{ik}) -\big (\nabla_{i} u R_{jk}-\nabla_{j} u R_{ik}).
\end{eqnarray*}
\end{lemma}
\begin{proof}
Rewriting the fundamental equation (\ref{fundamental equation}) in coordinates we get
\begin{eqnarray*}
m\nabla_{j}\nabla _{k}u=uR_{jk}-\lambda u g_{jk}.
\end{eqnarray*} From this we deduce
\begin{eqnarray*}
m\nabla_{i} \nabla_{j} \nabla_{k} u = \nabla_{i} u R_{jk}+u\nabla_{i}R_{jk}-\lambda \nabla_{i} u g_{jk},
\end{eqnarray*} which can rewrite succinctly as
\begin{eqnarray}\label{lemmaDifferenceeq1}
u\nabla_{i}R_{jk}=m\nabla_{i}\nabla_{j}\nabla_{k}u-\nabla_{i}uR_{jk}+\lambda \nabla_{i}u g_{jk}.
\end{eqnarray}

Similarly, it is not difficult to show that
\begin{eqnarray}\label{lemmaDifferenceeq2}
u\nabla_{j}R_{ik}=m\nabla_{j}\nabla_{i}\nabla_{k}u-\nabla_{j}uR_{ik}+\lambda\nabla_{j}ug_{ik}.
\end{eqnarray} Now we combine (\ref{lemmaDifferenceeq1}) and (\ref{lemmaDifferenceeq2}) jointly with Ricci identity in order to deduce
\begin{eqnarray*}
u\big( \nabla_{i} R_{jk}-\nabla_{j}R_{ik}\big)=mR_{ijkl}\nabla_{l}u+\lambda \big(\nabla_{i}ug_{jk}-\nabla_{j}ug_{ik}\big)-\big(\nabla_{i}uR_{jk}-\nabla_{j}uR_{ik}\big),
\end{eqnarray*} which finishes the proof of the lemma. 
\end{proof}

Before to state our next lemma, we introduce the $T$-tensor for $m$-quasi-Einstein manifolds. It was defined previously in \cite{Ranieri_Ernani}. 

\begin{eqnarray}\label{T-tensor}
T_{ijk}&=&\frac{m+n-2}{n-2}(R_{ik}\nabla_ju-R_{jk}\nabla_iu)+\frac{m}{n-2}(R_{jl}\nabla_lug_{ik}-R_{il}\nabla_lug_{jk})\nonumber\\
 &&+\frac{(n-1)(n-2)\lambda+mR}{(n-1)(n-2)}(\nabla_iug_{jk}-\nabla_jug_{ik})-\frac{u}{2(n-1)}(\nabla_iRg_{jk}-\nabla_jRg_{ik}).
\end{eqnarray}

With this notation we have the following lemma.

\begin{lemma}\label{CWT}
Let $\big(M^{n},g,u,\lambda\big)$ be a $m$-quasi-Einstein manifold. Then
\begin{equation*}
uC_{ijk}=mW_{ijkl}\nabla_lu+T_{ijk},
\end{equation*}
where $T$ is given by (\ref{T-tensor}).
\end{lemma}
\begin{proof}
Combining Eq. (\ref{cotton tensor}) with Lemma \ref{lemmaDifference} we obtain
\begin{eqnarray}\label{lemmaCWTeq1}
uC_{ijk}&=&u\big(\nabla_{i}R_{jk}-\nabla_{j}R_{ik}\big)-\frac{u}{2(n-1)}(\nabla_iRg_{jk}-\nabla_jRg_{ik})\nonumber\\
 &=&mR_{ijkl}\nabla_{l}u+\lambda \big (\nabla _{i}ug_{jk}-\nabla_{j}ug_{ik})-\big (\nabla_{i} u R_{jk}-\nabla_{j} u R_{ik})\\
 &&-\frac{u}{2(n-1)}(\nabla_iRg_{jk}-\nabla_jRg_{ik}).\nonumber
\end{eqnarray} Substituting (\ref{weyl tensor}) into (\ref{lemmaCWTeq1}) we get
\begin{eqnarray*}
uC_{ijk}&=&mW_{ijkl}\nabla_lu+\frac{m}{n-2}\big(R_{ik}g_{jl}+R_{jl}g_{ik}-R_{il}g_{jk}-R_{jk}g_{il}\big)\nabla_{l}u\nonumber\\
 &&-\frac{mR}{(n-1)(n-2)}\big(g_{jl}g_{ik}-g_{il}g_{jk}\big)\nabla_{l}u+\lambda\big(\nabla_{i}ug_{jk}-\nabla_{j}ug_{ik}\big)\nonumber\\
 &&-\big(\nabla_{i}u R_{jk}-\nabla_{j}uR_{ik}\big)-\frac{u}{2(n-1)}(\nabla_iRg_{jk}-\nabla_jRg_{ik})\nonumber\\
 &=&mW_{ijkl}\nabla_lu+\frac{m}{n-2}\big (R_{ik}\nabla_{j}u-R_{jk}\nabla_{i}u \big )+\frac{m}{n-2}\big(R_{jl}\nabla_lug_{ik}-R_{il}\nabla_lug_{jk}\big)\nonumber\\
 &&-\frac{mR}{(n-1)(n-2)}\big(g_{ik}\nabla_{j}u-g_{jk}\nabla_{i}u \big)+\lambda \big(\nabla_{i}u g_{jk}-\nabla_{j} u g_{ik}\big)\nonumber\\
 &&-\big(\nabla_{i} u R_{jk}-\nabla_{j}uR_{ik}\big)-\frac{u}{2(n-1)}(\nabla_iRg_{jk}-\nabla_jRg_{ik})\nonumber\\
 &=&mW_{ijkl}\nabla_lu+\frac{m+n-2}{n-2}(R_{ik}\nabla_{j}u-R_{jk}\nabla_{i}u\big)+\frac{m}{n-2}\big(R_{jl}\nabla_{l}ug_{ik}-R_{il}\nabla_{l}ug_{jk}\big)\nonumber\\
 &&+\frac{(n-1)(n-2)\lambda+mR}{(n-1)(n-2)}\big(\nabla_{i}u g_{jk}-\nabla_{j}u g_{ik}\big)-\frac{u}{2(n-1)}(\nabla_iRg_{jk}-\nabla_jRg_{ik})\\
 &=&mW_{ijkl}\nabla_lu+T_{ijk}.
\end{eqnarray*}This finishes the proof of the lemma.
\end{proof}

It should be point out that the tensor  $T$ has the same properties of the Cotton tensor, i.e., (\ref{simetC}) and (\ref{freetraceC}). Next, we going to compute the norm of the Cotton tensor. 

\begin{lemma}\label{lemmaCottonNorm}
Let $\big(M^{n},g,u,\lambda\big)$ be a $m$-quasi-Einstein manifold. Then we have
\begin{equation*}
u|C|^{2}=-\frac{2(m+n-2)}{n-2}C_{ijk}\mathring{R}_{jk}\nabla_{i}u+mC_{ijk}W_{ijkl}\nabla_lu.
\end{equation*}
\end{lemma}

\begin{proof} Initially, it follows by (\ref{simetC}) and (\ref{freetraceC}) that
\begin{eqnarray*}
-C_{ijk}\mathring{R}_{jk}\nabla_{i}u&=&-C_{ijk}R_{jk}\nabla_iu\nonumber\\
&=&-\dfrac{1}{2}\big(C_{ijk}-C_{jik}\big)R_{jk}\nabla_{i}u\nonumber\\
 &=&-\dfrac{1}{2}C_{ijk}R_{jk}\nabla_{i}u+\dfrac{1}{2}C_{ijk}R_{ik}\nabla_{j}u\nonumber\\
 &=&\dfrac{1}{2}C_{ijk}\big(R_{ik}\nabla_{j}u-R_{jk}\nabla_{i}u\big).
\end{eqnarray*}
This combined with (\ref{T-tensor}) provides
\begin{eqnarray}\label{lemmaCottonNormeq3}
-C_{ijk}\mathring{R_{jk}}\nabla_{i}u&=&\frac{n-2}{2(m+n-2)}C_{ijk}T_{ijk},
\end{eqnarray} where we used again the fact that the Cotton tensor is trace-free. Thus, by combining Lemma \ref{CWT} with (\ref{lemmaCottonNormeq3}) we arrive at
\begin{eqnarray*}
-C_{ijk}\mathring{R_{jk}}\nabla_{i}u&=&\frac{n-2}{2(m+n-2)}C_{ijk}\left(uC_{ijk}-mW_{ijkl}\nabla_lu\right)\\
	&=&\frac{n-2}{2(m+n-2)}\left(u|C|^{2}-mC_{ijk}W_{ijkl}\nabla_lu\right).
\end{eqnarray*} Consequently,
\begin{equation*}
u|C|^2=-\frac{2(m+n-2)}{n-2}C_{ijk}\mathring{R}_{jk}\nabla_iu+mC_{ijk}W_{ijkl}\nabla_lu,
\end{equation*} as we wanted to prove.
\end{proof}

Our next lemma provides a commutator formula for the Laplacian and Hessian acting on functions $h \in C^{4}(M).$ A detailed proof can be found in \cite[Proposition 7.1]{Viaclovsky} (see also \cite{Batista-Diogenes-Ranieri-Ribeiro JR}).

\begin{lemma}\label{lemma viackoviski}
Let $(M^n,\,g)$ be a Riamannian manifold and $h \in C^{4}(M).$ Then we have
\begin{eqnarray*}
(\Delta \nabla ^{2}h)_{ij}&=&\nabla ^{2}_{ij}\Delta h+(R_{jp}g_{ik}+R_{ip}g_{jk}-2R_{ikjp})\nabla_{k}\nabla_{p}h \\
	&+& (\nabla_{i}R_{jp}+\nabla_{j}R_{pi}-\nabla _{p}R_{ij})\nabla_{p}h,
\end{eqnarray*} where $\nabla ^{2}$ represents the Hessian.
\end{lemma}

We shall use Lemma \ref{lemma viackoviski} to obtain the following result that will be used in the proof of Theorem \ref{th3}.

\begin{lemma}\label{lemmaLaplacianoRic}
Let $\big(M^{n},g,u,\lambda\big)$ be a $m$-quasi-Einstein manifold. Then it holds
\begin{eqnarray*}
u\langle\Delta\Rc,\Rc\rangle&=&\frac{2Ru}{n}|\Rc|^2+u\langle\nabla^2R,\Rc\rangle+\frac{2(n+m)}{n}\Rc(\nabla u,\nabla R)+2m\nabla_{i}\R_{jp}\R_{ij}\nabla_{p}u\nonumber\\
	&&+2u\left(\R_{ip}\R_{kp}\R_{ik}-R_{ikjp}\R_{ij}\R_{kp}\right)-\frac{m+2}{2}\langle\nabla|\Rc|^2,\nabla u\rangle.
\end{eqnarray*}
\end{lemma}
\begin{proof} Computing the Laplacian of (\ref{traceless Ricci}) we infer
\begin{eqnarray}\label{lemmaLaplacianoRiceq1}
m\big(\Delta \mathring{\nabla^2}u\big)_{ij}=\Delta u \mathring{R}_{ij}+u(\Delta\mathring{Ric})_{ij}+2\nabla_k\mathring{R_{ij}}\nabla_ku.
\end{eqnarray} Hence, we use (\ref{lemmaLaplacianoRiceq1}) and (\ref{laplaciano}) to obtain
\begin{eqnarray}\label{lemmaLaplacianoRiceq2}
u\langle \Delta \mathring{Ric}, \mathring{Ric}\rangle&=&u(\Delta\mathring{Ric})_{ij}\R_{ij}\nonumber\\
 &=&\left[m\big(\Delta \mathring{\nabla^2}u\big)_{ij}-\Delta u \mathring{R_{ij}}-2\nabla_k\mathring{R_{ij}}\nabla_ku\right]\R_{ij}\nonumber \\ 
&=&m\left((\Delta\nabla^2u)_{ij}-\frac{\Delta(\Delta u)}{n}g_{ij}\right)\R_{ij}-\dfrac{u}{m}\big(R-n\lambda\big)|\Rc|^2-2\nabla_k\mathring{R_{ij}}\R_{ij}\nabla_ku\nonumber\\
&=&m(\Delta\nabla^2u)_{ij}\R_{ij}-\dfrac{u}{m}\big(R-n\lambda\big)|\Rc|^2-\langle\nabla|\mathring{Ric}|^2,\nabla u\rangle.
\end{eqnarray} From Lemma \ref{lemma viackoviski} and by (\ref{laplaciano}) we have
\begin{eqnarray*}
(\Delta\nabla^2u)_{ij}\R_{ij}&=&\left[\nabla ^{2}_{ij}\Delta u+(R_{jp}g_{ik}+R_{ip}g_{jk}-2R_{ikjp})\nabla_{k}\nabla_{p}u\right.\nonumber\\
	&&+\left.(\nabla_{i}R_{jp}+\nabla_{j}R_{pi}-\nabla _{p}R_{ij})\nabla_{p}u\right]\R_{ij}\nonumber\\
	&=&\frac{1}{m}\left[(R-n\lambda)\nabla_{i}\nabla_{j}u+u\nabla_i\nabla_jR+\nabla_iu\nabla_jR+\nabla_ju\nabla_iR\right]\R_{ij}\nonumber\\
	&&+(R_{jp}\R_{jk}+R_{ip}\R_{ki}-2R_{ikjp}\R_{ij})\nabla_{k}\nabla_{p}u\nonumber\\
	&&+ \big (\nabla_{i}R_{jp}\R_{ij}+\nabla_{j}R_{pi}\R_{ij}-\nabla _{p}R_{ij}\R_{ij}\big)\nabla_{p}u\nonumber\\
	&=&\frac{R-n\lambda}{m}\langle\nabla^2u,\Rc\rangle+\frac{u}{m}\langle\nabla^2R,\Rc\rangle+\frac{2}{m}\Rc(\nabla u,\nabla R)\nonumber\\
	&&+2(R_{ip}\R_{ik}-R_{ikjp}\R_{ij})\nabla_k\nabla_pu+2\nabla_{i}R_{jp}\R_{ij}\nabla_{p}u-\nabla_p\R_{ij}\R_{ij}\nabla_pu.
\end{eqnarray*} Substituting (\ref{fundamental equation}) into the above expression we get
\begin{eqnarray*}
(\Delta\nabla^2u)_{ij}\R_{ij}&=&\frac{(R-n\lambda)u}{m^2}\langle Ric-\lambda g,\Rc\rangle+\frac{u}{m}\langle\nabla^2R,\Rc\rangle+\frac{2}{m}\Rc(\nabla u,\nabla R)\nonumber\\
	&&+\frac{2u}{m}(R_{ip}\R_{ik}-R_{ikjp}\R_{ij})(R_{kp}-\lambda g_{kp})+2\nabla_{i}R_{jp}\R_{ij}\nabla_{p}u-\frac{1}{2}\langle\nabla|\Rc|^2,\nabla u\rangle\nonumber\\
	&=&\frac{(R-n\lambda)u}{m^2}|\Rc|^2+\frac{u}{m}\langle\nabla^2R,\Rc\rangle+\frac{2}{m}\Rc(\nabla u,\nabla R)\nonumber\\
	&&+\frac{2u}{m}(R_{ip}R_{kp}\R_{ik}-R_{ikjp}\R_{ij}R_{kp})+2\nabla_{i}R_{jp}\R_{ij}\nabla_{p}u-\frac{1}{2}\langle\nabla|\Rc|^2,\nabla u\rangle\nonumber.
\end{eqnarray*} Since $\R_{ij}=R_{ij}-\dfrac{R}{n}g_{ij}$ we deduce
\begin{eqnarray*}
(\Delta\nabla^2u)_{ij}\R_{ij}&=&\frac{(R-n\lambda)u}{m^2}|\Rc|^2+\frac{u}{m}\langle\nabla^2R,\Rc\rangle+\frac{2}{m}\Rc(\nabla u,\nabla R)\nonumber\\
	&&+\frac{2u}{m}\left[(\R_{ip}+\frac{R}{n}g_{ip})(\R_{kp}+\frac{R}{n}g_{kp})\R_{ik}-R_{ikjp}\R_{ij}\R_{kp}-\frac{R}{n}R_{ij}\R_{ij}\right]\nonumber\\
	&&+2\nabla_{i}\R_{jp}\R_{ij}\nabla_{p}u+\frac{2}{n}\R_{ij}\nabla_iR\nabla_ju-\frac{1}{2}\langle\nabla|\Rc|^2,\nabla u\rangle\nonumber\\
	&=&\frac{(R-n\lambda)u}{m^2}|\Rc|^2+\frac{u}{m}\langle\nabla^2R,\Rc\rangle+\frac{2}{m}\Rc(\nabla u,\nabla R)\nonumber\\
	&&+\frac{2u}{m}\left[\R_{ip}\R_{kp}\R_{ik}+\frac{2R}{n}|\Rc|^2-R_{ikjp}\R_{ij}\R_{kp}-\frac{R}{n}|\Rc|^2\right]\nonumber\\
	&&+2\nabla_{i}\R_{jp}\R_{ij}\nabla_{p}u+\frac{2}{n}\Rc(\nabla R,\nabla u)-\frac{1}{2}\langle\nabla|\Rc|^2,\nabla u\rangle\nonumber\\
	&=&\frac{[(n+2m)R-n^2\lambda]u}{nm^2}|\Rc|^2+\frac{u}{m}\langle\nabla^2R,\Rc\rangle+\frac{2(n+m)}{nm}\Rc(\nabla u,\nabla R)\nonumber\\
	&&+\frac{2u}{m}\left(\R_{ip}\R_{kp}\R_{ik}-R_{ikjp}\R_{ij}\R_{kp}\right)+2\nabla_{i}\R_{jp}\R_{ij}\nabla_{p}u-\frac{1}{2}\langle\nabla|\Rc|^2,\nabla u\rangle\nonumber.
\end{eqnarray*} Using this last expression into (\ref{lemmaLaplacianoRiceq2}) yields
\begin{eqnarray*}
u\langle \Delta \mathring{Ric}, \mathring{Ric}\rangle&=&\frac{[(n+2m)R-n^2\lambda]u}{nm}|\Rc|^2+u\langle\nabla^2R,\Rc\rangle+\frac{2(n+m)}{n}\Rc(\nabla u,\nabla R)\nonumber\\
	&&+2u\left(\R_{ip}\R_{kp}\R_{ik}-R_{ikjp}\R_{ij}\R_{kp}\right)+2m\nabla_{i}\R_{jp}\R_{ij}\nabla_{p}u\nonumber\\
	&&-\frac{m}{2}\langle\nabla|\Rc|^2,\nabla u\rangle-\dfrac{u}{m}\big(R-n\lambda\big)|\Rc|^2-\langle\nabla|\mathring{Ric}|^2,\nabla u\rangle\nonumber\\
	&=&\frac{2Ru}{n}|\Rc|^2+u\langle\nabla^2R,\Rc\rangle+\frac{2(n+m)}{n}\Rc(\nabla u,\nabla R)+2m\nabla_{i}\R_{jp}\R_{ij}\nabla_{p}u\nonumber\\
	&&+2u\left(\R_{ip}\R_{kp}\R_{ik}-R_{ikjp}\R_{ij}\R_{kp}\right)-\frac{m+2}{2}\langle\nabla|\Rc|^2,\nabla u\rangle.
\end{eqnarray*} So, the proof is finished.

\end{proof}

\section{The boundary of a quasi-Einstein manifold}\label{secBoundary}

In this section we will study a compact $m$-quasi-Einstein manifold with no empty boundary. To begin with, we remember that $u>0$ in the interior of $M^n$ and $u=0$ on the boundary $\partial M.$ Thus, $N=-\frac{\nabla u}{|\nabla u|}$ is the outward unit vector. Besides, on the boundary, we already know that
\begin{equation}\label{HessianBoundary}
\nabla^2u=0.
\end{equation} Next, for any $X\in\mathfrak{X}(\partial M),$ we have
\begin{eqnarray*}
X(|\nabla u|^{2})=2\langle \nabla _{X}\nabla u, \nabla u \rangle =2\nabla^{2}u(X,\nabla u)=0.
\end{eqnarray*} Therefore, $|\nabla u|\neq0$ is constant along $\partial M$ (for more details, see Propositions 2.2 and 2.3 of \cite{He-Petersen-Wylie2012}).  From now on, we consider an orthonormal frame $\{e_1,\ldots,e_{n-1},e_n=-\frac{\nabla u}{|\nabla u|}\}.$ Thus from (\ref{HessianBoundary}) the second fundamental form at $\partial M$ is given by
\begin{eqnarray}\label{fundamental second}
h_{ij}=\langle \nabla _{e_{i}}N, e_{j}\rangle= -\frac{1}{|\nabla u|}\nabla _{i}\nabla _{j} u=0,
\end{eqnarray} for any $1\leq i,j\leq n-1.$ Hence, $\partial M$ with induced metric is totally geodesic. In particular, by Gauss equation $\big(R^{\partial M}_{ijkl}=R_{ijkl}-h_{il}h_{jk}+h_{ik}h_{jl}\big)$ we immediately infer
\begin{eqnarray}
 \label{GaussEq}
R^{\partial M}_{ijkl}&=&R_{ijkl},
\end{eqnarray}

\begin{eqnarray}
\label{RicciBordoM}
R^{\partial M}_{ik}&=&R_{ik}-R_{inkn}
\end{eqnarray}
and
\begin{eqnarray}
\label{EscalarBordoM}
R^{\partial M}&=&R-2R_{nn}.
\end{eqnarray}

Before to show our next lemmas we remember a well-known result that will be useful in this paper.

\begin{lemma}\label{lemmadivtensor}
Let $H$ be a $(0,2)$-tensor on a Riemannian manifold $(M^{n},g)$ then 
\begin{equation*}
\div(H(\varphi Z)) = \varphi(\div \, H)(Z)+\varphi \langle\nabla Z, H\rangle + H(\nabla \varphi, Z),
\end{equation*}
for all $Z$ $\in \mathfrak{X}(M)$ and any smooth function $\varphi$ on $M.$	
\end{lemma}

Another fact that we will useful is a combination between twice-contracted second Bianchi identity and the traceless Ricci tensor given by
\begin{equation}\label{divtracelessRic}
\div\mathring{Ric}=\frac{n-2}{2n}\nabla R.
\end{equation}

Now, we are ready to state our next lemma.

\begin{lemma}\label{lemma1111}
Let $\big(M^{n},g,u,\lambda\big)$ be a compact, oriented $m$-quasi-Einstein manifold with connected boundary. Then we have
\begin{equation*}
\dfrac{1}{m}\int_{M} u|\mathring{Ric}|^{2}dM=-\dfrac{1}{\kappa} \int_{\partial M}\mathring{Ric}(\nabla u, \nabla u)dS-\dfrac{(n-2)}{2n}\int_{M}\langle\nabla R, \nabla u \rangle dM,
\end{equation*}
where $\kappa=|\nabla u|_{|_{\partial M}}.$	
\end{lemma}
\begin{proof} 
Firstly, we choose $H=\mathring{Ric},$ $Z=\nabla u$ and $\varphi=1$ in Lemma \ref{lemma1111} and we combine with Eqs. (\ref{fundamental equation}) and (\ref{divtracelessRic}) to obtain
\begin{eqnarray*}
\div(\mathring{Ric}(\nabla u))&=&(\div\mathring{Ric})(\nabla u)+\langle \mathring{Ric}, \nabla ^{2}u\rangle\\ 
 &=&\dfrac{(n-2)}{2n}\langle \nabla R, \nabla u \rangle + \langle \mathring{Ric},\dfrac{u}{m}[\mathring{Ric}+(\dfrac{R}{n}-\lambda)g]\rangle \\ 
 &=& \dfrac{(n-2)}{2n}\langle \nabla R, \nabla u\rangle +\dfrac{u}{m}|\mathring{Ric}|^{2}.
	\end{eqnarray*}
Integrating this expression over $M$ and using Stokes' formula we deduce
\begin{eqnarray*}
\dfrac{1}{m}\int_{M} u| \mathring{Ric}|^{2}dM &=&\int_{M}div(\mathring{Ric}(\nabla u))dM-\dfrac{(n-2)}{2n}\int_{M}\langle \nabla R, \nabla u \rangle dM \\
 &=& \int_{\partial M} \langle \mathring{Ric}(\nabla u),N\rangle dS-\dfrac{(n-2)}{2n}\int_{M}\langle \nabla R, \nabla u \rangle dM \\ 
 &=&-\dfrac{1}{|\nabla u|}_{|_{\partial M}} \int_{\partial M}\mathring{Ric}(\nabla u, \nabla u)dS-\dfrac{(n-2)}{2n}\int_{M}\langle \nabla R, \nabla u \rangle dM.
\end{eqnarray*} This finishes the proof of the lemma.
\end{proof}

The next lemma gives an expression for the integral of the scalar curvature on the boun\-dary.

\begin{lemma}\label{lemmaRboundary}
Let $\big(M^{n},g,u,\lambda\big)$ be a compact, oriented $m$-quasi-Einstein manifold with connected boundary. Then we have
\begin{equation*}
\int_{\partial M}R^{\partial M}dS=\dfrac{2}{m\kappa}\int_{M}u|\mathring{Ric}|^{2}dM-\frac{n-2}{n\kappa}\int_{M}R\Delta udM,
\end{equation*} where $\kappa=|\nabla u|_{|_{\partial M}}$ is a constant.
\end{lemma}
\begin{proof} It is not difficult to check from (\ref{EscalarBordoM}) that $$\mathring{Ric}(N,N)=R_{nn}-\frac{R}{n}$$ and therefore, we have $$\frac{1}{\kappa^2}\mathring{Ric}(\nabla u,\nabla u)=\frac{R-R^{\partial M}}{2}-\frac{R}{n},$$ so that 
\begin{eqnarray}\label{lemmaRboundaryeq1}
 \mathring{Ric}(\nabla u,\nabla u)&=&\kappa^2\left(\frac{n-2}{2n}R-\dfrac{R^{\partial M}}{2}\right).
\end{eqnarray}
Replacing (\ref{lemmaRboundaryeq1}) into Lemma \ref{lemma1111} we obtain
\begin{eqnarray*}
\dfrac{1}{m}\int_{M} u| \mathring{Ric}|^{2}dM &=&-\kappa\int_{\partial M}\left(\frac{n-2}{2n}R-\dfrac{R^{\partial M}}{2}\right)dS-\dfrac{n-2}{2n}\int_{M}\langle \nabla R, \nabla u \rangle dM,
\end{eqnarray*}
Rearranging this expression and using Stokes'  formula we have
\begin{eqnarray*}
\dfrac{\kappa}{2}\int_{\partial M}R^{\partial M}dS&=&\dfrac{1}{m}\int_{M}u|\mathring{Ric}|^{2}dM+\dfrac{(n-2)\kappa}{2n}\int_{\partial M}R\,dS\\
 &&+\dfrac{n-2}{2n}\left(\int_{\partial M}R\langle\nabla u,N\rangle dS-\int_{M}R\Delta udM\right)\\
 &=&\dfrac{1}{m}\int_{M}u|\mathring{Ric}|^{2}dM+\dfrac{(n-2)\kappa}{2n}\int_{\partial M}R\,dS\\
 &&-\dfrac{(n-2)\kappa}{2n}\int_{\partial M}RdS-\frac{n-2}{2n}\int_{M}R\Delta udM\\
 &=&\dfrac{1}{m}\int_{M}u|\mathring{Ric}|^{2}dM-\frac{n-2}{2n}\int_{M}R\Delta udM,
\end{eqnarray*}
as we wanted to prove.
\end{proof}

To finish this section we going to compute the Euler characteristic of the boundary $\partial M$ of a $5$-dimensional compact quasi-Einstein manifold $(M^{5},g,u,\lambda).$ Similar result was obtained previously for $V$-static spaces in \cite{Barbosa-Lima-Freitas}.

\begin{proposition}\label{propGauss-Bonnet-Chern}
Let $(M^{5},g,u,\lambda)$ be a compact quasi-Einstein whose the boundary is a connected Einstein manifold. Suppose that the scalar curvature satisfies $R\leq 5\lambda.$ Then it holds
\begin{equation}
8\pi^2\chi(\partial M) \geq \frac{3}{200}R_{min}^{2}|\partial M|.
\end{equation} Moreover, the equality holds if and only if  $(M^{5},g)$ is isometric, up to scaling, to $\mathbb{S}^{5}_{+}.$
\end{proposition}
\begin{proof}
Since $R\leq 5\lambda$ we infer that $\Delta u = \frac{u}{m}(R-5\lambda)\leq 0.$  Taking into account that the scalar curvature is positive, we have $R_{min}\Delta u \geq R \Delta u.$ Therefore, it follows that
\begin{eqnarray*}
\int_MR\Delta udM&\leq&R_{min}\int_M\Delta udM=R_{min}\int_{\partial M}\langle\nabla u,N\rangle dS\nonumber\\
	&=&-\kappa R_{min}|\partial M|.
\end{eqnarray*} Combining this with Lemma \ref{lemmaRboundary} we obtain
\begin{eqnarray*}
\int_{\partial M}R^{\partial M}dS&\geq &\frac{2}{m\kappa}\int_{M}u|\mathring{Ric}|^{2}dM+\frac{3}{5} R_{min}|\partial M| \nonumber \\ 
 &\geq& \frac{3}{5} R_{min}|\partial M|.
\end{eqnarray*} This jointly with the H\"older-inequality we infer
\begin{eqnarray}\label{propGauss-Bonnet-Cherneq1}
\frac{9}{25}R_{min}^{2}|\partial M|^2\leq \Big(\int_{\partial M}R^{\partial M}dS\Big)^{2}\leq\int_{\partial M} (R^{\partial M})^{2}dS|\partial M|.
\end{eqnarray} Now, by using the Gauss-Bonnet-Chern's formula and the fact that the boundary of $M$ is an Einstein manifold we deduce 
\begin{eqnarray}\label{propGauss-Bonnet-Cherneq2}
8\pi ^{2}\chi(\partial M)&=&\frac{1}{4}\int_{\partial M}|W^{\partial M}|^{2}dS + \dfrac{1}{24}\int_{\partial M} (R^{\partial M})^{2}dS\nonumber\\
 &\geq& \frac{3}{200}R_{min}^{2}|\partial M|.
\end{eqnarray} Finaly, the equality holds in (\ref{propGauss-Bonnet-Cherneq2}) if and only if the equality also holds in (\ref{propGauss-Bonnet-Cherneq1}). This implies that $(M^{5},g)$ is an Einstein manifold and therefore, the result follows by Proposition 2.4 of \cite{He-Petersen-Wylie2012}.	So, the proof is finished.
\end{proof}

\section{Proof of the  main results}\label{secProofs}

\subsection{Proof of Theorem \ref{th1}}

\begin{proof} 	
The first part of the proof is standard, and it follows the same steps of Theorem 1.5 of \cite{BS} (see also Theorem 2 in \cite{Batista-Diogenes-Ranieri-Ribeiro JR}). We include it here for sake of completeness. Initially, we have
\begin{eqnarray*}
Ric^{\partial M}(V,V)&\geq&\dfrac{R^{\partial M}}{n-1}g_{\partial M}(V,V) \nonumber \\
&=&\dfrac{R^{\partial M}}{n-1},
\end{eqnarray*} where $V\in T \partial M$ with $|V|_{g}=1$. Since $\inf _{\partial M} R^{\partial M} > 0$ there is a $\delta >0$ such that
\begin{eqnarray*}
R(\partial M, g_{\partial M})&=&\inf\big\{Ric^{\partial M}(V,V)/ V\in T \partial M, |V|_{g}=1\big\} \nonumber \\
&=&(n-2)\delta.
\end{eqnarray*} Therefore, we have
\begin{eqnarray*}
Ric^{\partial M}\geq (n-2)\delta g_{\partial M}
\end{eqnarray*} and consequently, by Bonnet-Myers Theorem we arrive at
\begin{eqnarray*}
diam_{g_{\partial M}}(\partial M)\leq \dfrac{\pi}{\sqrt{\delta}}.
\end{eqnarray*} Now, by Bishop-Gromov Theorem it follows that
\begin{eqnarray*}
|B_{\rho}(p)|\leq |B^{\delta}_{\rho}(p)|,
\end{eqnarray*} with $\rho \leq \dfrac{\pi}{\sqrt{\delta}}$, $p \in \partial M$ and $|B^{\delta}_{\rho}(p)|$ denotes the volume of a ball in a sphere $S^{n-1}$ with the metric $g_{\delta}=\delta ^{-1}g_{can}$ and radius $\rho$. In particular, we infer
\begin{eqnarray}\label{th1eq01}
|B_{\frac{\pi}{\sqrt{\delta}}}(p)|\leq |S^{n-1}|_{g_{\delta}}.
\end{eqnarray} Taking into account that $|S^{n-1}|_{g_{\delta}}=\delta^{-\dfrac{(n-1)}{2}} \omega_{n-1}$ and (\ref{th1eq01}) we obtain
\begin{eqnarray*}
|\partial M | &\leq& |B_{\dfrac{\pi}{\sqrt{\delta}}}(p)| \nonumber \\
&\leq& |S^{n-1}|g_{\delta} \nonumber \\
&=&\delta ^{-\dfrac{(n-1)}{2}}\omega_{n-1},
\end{eqnarray*} which we can rewrite succinctly as
\begin{eqnarray}\label{th1eq02}
\big(\omega_{n-1}\big)^{\dfrac{2}{n-1}}\geq |\partial M |^{\dfrac{2}{n-1}}\delta.
\end{eqnarray}

Since $\partial M$ is compact, there exists $X \in T\partial M$ of unit norm and such that 
\begin{eqnarray}\label{th1eq03}
(n-2)\delta &=& R(\partial M, g_{\partial M}) \nonumber \\
&=& Ric^{\partial M}(X,X).
\end{eqnarray} Notice also that
\begin{eqnarray}\label{th1eq04}
Ric^{\partial M}\geq \dfrac{R^{\partial M}}{n-1}g_{\partial M}.
\end{eqnarray} Hence, combining (\ref{th1eq03}) and (\ref{th1eq04}) implies 
\begin{eqnarray*}
(n-2)\delta \geq \dfrac{R^{\partial M}}{n-1},
\end{eqnarray*}
consequently,
\begin{eqnarray}\label{th1eq05}
\delta \geq \dfrac{R^{\partial M}}{(n-1)(n-2)}.
\end{eqnarray} Thus, substituting (\ref{th1eq05}) into (\ref{th1eq02}) we arrive at
\begin{eqnarray*}
\big(\omega_{n-1}\big)^{\dfrac{2}{n-1}}\geq |\partial M|^{\dfrac{2}{n-1}}\dfrac{R^{\partial M}}{(n-1)(n-2)}.
\end{eqnarray*} Thus, we have
\begin{eqnarray*}
(n-1)(n-2)\big (\omega_{n-1}\big )^{\frac{2}{n-1}}\geq R^{\partial M}|\partial M|^{\frac{2}{n-1}}. 
\end{eqnarray*} Next, upon integrating this expression over $\partial M$ we obtain
\begin{eqnarray*}
(n-1)(n-2)(\omega_{n-1})^\frac{2}{n-1}|\partial M|&\geq&|\partial M|^\frac{2}{n-1}\int_{\partial M}R^{\partial M}dS,
\end{eqnarray*} 
so that
\begin{eqnarray*}
(n-1)(n-2)(\omega_{n-1})^\frac{2}{n-1}|\partial M|^{\frac{n-3}{n-1}}&\geq&\,\int_{\partial M}R^{\partial M}dS.
\end{eqnarray*} Now, it suffices to invoke Lemma \ref{lemmaRboundary} to infer
\begin{equation}\label{th1eq1}
(n-1)(n-2)(\omega_{n-1})^\frac{2}{n-1}|\partial M|^{\frac{n-3}{n-1}}\geq\dfrac{2}{m\kappa}\int_{M}u|\mathring{Ric}|^{2}dM-\frac{n-2}{n\kappa}\int_{M}R\Delta udM.
\end{equation}

Proceeding, using that $R\leq n\lambda,$ we obtain $\Delta u=\frac{u}{m}(R-n\lambda)\leq0.$ Therefore, we deduce $$R_{min}\Delta u\geq R\Delta u$$ (we remark that in general $R\geq\frac{n(n-1)}{m+n-1}\lambda,$ see Remark 5.1 of \cite{He-Petersen-Wylie2012}). Thus, we obtain
\begin{eqnarray}\label{th1eq2}
\int_MR\Delta udM&\leq&R_{min}\int_M\Delta udM=R_{min}\int_{\partial M}\langle\nabla u,N\rangle dS\nonumber\\
 &=&-\kappa R_{min}|\partial M|.
\end{eqnarray}
Replacing (\ref{th1eq2}) into (\ref{th1eq1}) we deduce
\begin{eqnarray}\label{th1eq3}
(n-1)(n-2)(\omega_{n-1})^\frac{2}{n-1}|\partial M|^{\frac{n-3}{n-1}}&\geq&\dfrac{2}{m\kappa}\int_{M}u|\mathring{Ric}|^{2}dM+\frac{n-2}{n}R_{min}|\partial M|\nonumber\\
 &\geq&\frac{n-2}{n}R_{min}|\partial M|.
\end{eqnarray} This can be rewrite as follows
\begin{equation}\label{th1eq4}
|\partial M|\leq{\omega_{n-1}}\left(\frac{n(n-1)}{R_{min}}\right)^{\frac{n-1}{2}}.
\end{equation} Finally, the equality holds in (\ref{th1eq4}) if and only if the equality holds in (\ref{th1eq3}) and this therefore implies $\mathring{Ric}=0,$ that is, $(M,\,g)$ is an Einstein manifold. So, it suffices to apply Proposition 2.4 of \cite{Petersen and Chenxu} to conclude that $(M,g)$ is the hemisphere $\Bbb{S}^n_+.$ This finishes the proof of the theorem. 
\end{proof}

\subsection{Proof of Theorem \ref{th2}}

\begin{proof}
Since $R\leq 3\lambda,$ then $\Delta u=\frac{u}{m}(R-3\lambda)\leq0.$ Therefore, we obtain $R_{min}\Delta u\geq R\Delta u.$ Thus, it follows that
\begin{eqnarray*}
\int_MR\Delta udM&\leq&R_{min}\int_M\Delta udM=R_{min}\int_{\partial M}\langle\nabla u,N\rangle dS\nonumber\\
 &=&-\kappa R_{min}|\partial M|.
\end{eqnarray*} Combining this with Lemma \ref{lemmaRboundary} we infer
\begin{equation}\label{th2eq1}
\int_{\partial M}R^{\partial M}dS\geq\frac{2}{m\kappa}\int_Mu|\mathring{Ric}|^2dM+\frac{R_{min}}{3}|\partial M|.
\end{equation} We remember that in dimension three we have $R^{\partial M}=2K,$ where $K$ is the Gaussian curvature of $\partial M.$ Therefore, by (\ref{th2eq1}) we obtain
\begin{equation*}
2\int_{\partial M}KdS\geq\frac{2}{m\kappa}\int_Mu|\mathring{Ric}|^2dM+\frac{R_{min}}{3}|\partial M|
\end{equation*} and hence, we deduce 
\begin{equation}\label{th2eq2}
\int_{\partial M}KdS\geq\frac{R_{min}}{6}|\partial M|>0
\end{equation} Next, by Gauss-Bonnet's Theorem we conclude that $\partial M$ is a $2$-sphere and thus $$\int_{\partial M} KdS = 4\pi.$$ Finally, by (\ref{th2eq2}) we arrive at
\begin{equation}\label{th2eq3}
|\partial M|\leq\dfrac{24\pi}{R_{min}}.
\end{equation} Moreover, the equality holds if and only if the equality holds in (\ref{th2eq2}). In this case, $(M^{3},g)$ is an Einstein manifold and by Proposition 2.4 of \cite{Petersen and Chenxu}, $(M^{3},\,g)$ must be isometric to $\Bbb{S}^3_+,$ as asserted.  
\end{proof}

\subsection{Proof of Theorem \ref{th3}}

\begin{proof}
To begin with, notice that 
\begin{eqnarray*}
\frac{1}{2}\div \big(u\nabla|\mathring{Ric}|^{2}\big)&=&\frac{1}{2}\nabla_k\big(u\nabla_k(\mathring{R}_{ij})^2\big)=\nabla_k(u\mathring{R}_{ij}\nabla_k\mathring{R}_{ij})\nonumber\\
 &=&\nabla_k\R_{ij}\R_{ij}\nabla_ku+u\nabla_k\R_{ij}\nabla_k\R_{ij}+u\R_{ij}\nabla_k\nabla_k\R_{ij}\nonumber\\
&=&\frac{1}{2}\langle \nabla|\mathring{Ric}|^2,\nabla u\rangle +u\langle \Delta \mathring{Ric}, \mathring{Ric} \rangle + u |\nabla\mathring{Ric}|^{2}.
	\end{eqnarray*} This jointly with Lemma \ref{lemmaLaplacianoRic} we deduce
\begin{eqnarray}\label{th3eq1}
\frac{1}{2}\div \big(u\nabla|\mathring{Ric}|^{2}\big)&=&\frac{1}{2}\langle \nabla|\mathring{Ric}|^2,\nabla u\rangle+u|\nabla\mathring{Ric}|^{2}+\frac{2Ru}{n}|\Rc|^2+u\langle\nabla^2R,\Rc\rangle\nonumber\\
 &&+\frac{2(n+m)}{n}\Rc(\nabla u,\nabla R)+2m\nabla_{i}\R_{jp}\R_{ij}\nabla_{p}u-\frac{m+2}{2}\langle\nabla|\Rc|^2,\nabla u\rangle\nonumber\\
	&&+2u\R_{ip}\R_{kp}\R_{ik}-2uR_{ikjp}\R_{ij}\R_{kp}\nonumber\\
	&=&u|\nabla\mathring{Ric}|^{2}+\frac{2Ru}{n}|\Rc|^2+u\langle\nabla^2R,\Rc\rangle+\frac{2(n+m)}{n}\Rc(\nabla u,\nabla R)\nonumber\\
	&&+2m\nabla_{i}\R_{jp}\R_{ij}\nabla_{p}u-\frac{m+1}{2}\langle\nabla|\Rc|^2,\nabla u\rangle+2u\R_{ip}\R_{kp}\R_{ik}\nonumber\\
	&&-2uR_{ikjp}\R_{ij}\R_{kp}
	\end{eqnarray}
On the other hand, by (\ref{cotton tensor}) and the fact that the Cotton tensor is skew-symmetric in the first two indices we infer
\begin{eqnarray*}
\nabla_{i}\R_{jp}\R_{ij}\nabla_{p}u&=&\left(\nabla_iR_{jp}-\frac{1}{n}\nabla_iRg_{jp}\right)\R_{ij}\nabla_pu\nonumber\\
 &=&\left(C_{ipj}+\nabla_pR_{ij}+\frac{1}{2(n-1)}(\nabla_iRg_{pj}-\nabla_pRg_{ij})\right)\R_{ij}\nabla_pu-\frac{1}{n}\Rc(\nabla R,\nabla u)\nonumber\\
 &=&C_{ipj}\R_{ij}\nabla_pu+\nabla_pR_{ij}\R_{ij}\nabla_pu+\frac{1}{2(n-1)}\Rc(\nabla R,\nabla u)-\frac{1}{n}\Rc(\nabla R,\nabla u)\nonumber\\
 &=&-C_{pij}\R_{ij}\nabla_pu+\frac{1}{2}\langle\nabla|\Rc|^2,\nabla u\rangle-\frac{n-2}{2n(n-1)}\Rc(\nabla R,\nabla u).
\end{eqnarray*} This combined with Lemma \ref{lemmaCottonNorm} yields
\begin{eqnarray}\label{th3eq2}
\nabla_{i}\R_{jp}\R_{ij}\nabla_{p}u&=&\frac{(n-2)u}{2(m+n-2)}|C|^2-\frac{m(n-2)}{2(m+n-2)}C_{ijk}W_{ijkl}\nabla_lu\nonumber\\
 &&+\frac{1}{2}\langle\nabla|\Rc|^2,\nabla u\rangle-\frac{n-2}{2n(n-1)}\Rc(\nabla R,\nabla u).
\end{eqnarray}

Proceeding, it follows from (\ref{weyl tensor}) that
\begin{eqnarray}\label{th3eq3}
R_{ikjp}\R_{ij}\R_{kp}&=&W_{ikjp}\R_{ij}\R_{kp}+\frac{1}{n-2}(R_{ij}g_{kp}+R_{kp}g_{ij}-R_{ip}g_{kj}-R_{kj}g_{ip})\R_{ij}\R_{pk}\nonumber\\
 &&-\frac{R}{(n-1)(n-2)}(g_{kp}g_{ij}-g_{ip}g_{kj})\R_{ij}\R_{kp}\nonumber\\
 &=&W_{ikjp}\R_{ij}\R_{kp}-\frac{2}{n-2}R_{ip}\R_{ij}\R_{pj}+\frac{R}{(n-1)(n-2)}|\Rc|^2\nonumber\\
 &=&W_{ikjp}\R_{ij}\R_{kp}-\frac{2}{n-2}\R_{ip}\R_{ij}\R_{pj}-\frac{2R}{n(n-2)}|\Rc|^2+\frac{R}{(n-1)(n-2)}|\Rc|^2\nonumber\\
 &=&W_{ikjp}\R_{ij}\R_{kp}-\frac{2}{n-2}\R_{ip}\R_{ij}\R_{pj}-\frac{R}{n(n-1)}|\Rc|^2.
\end{eqnarray}
Replacing (\ref{th3eq2}) and (\ref{th3eq3}) into (\ref{th3eq1}) we obtain
\begin{eqnarray}
\frac{1}{2}\div \big(u\nabla|\mathring{Ric}|^{2}\big)&=&u|\nabla\mathring{Ric}|^{2}+\frac{2Ru}{n}|\Rc|^2+u\langle\nabla^2R,\Rc\rangle+\frac{2(n+m)}{n}\Rc(\nabla u,\nabla R)\nonumber\\
	&&+\frac{m(n-2)u}{m+n-2}|C|^2-\frac{m^2(n-2)}{m+n-2}C_{ijk}W_{ijkl}\nabla_lu+m\langle\nabla|\Rc|^2,\nabla u\rangle\nonumber\\
	&&-\frac{m(n-2)}{n(n-1)}\Rc(\nabla R,\nabla u)-\frac{m+1}{2}\langle\nabla|\Rc|^2,\nabla u\rangle+2u\R_{ip}\R_{kp}\R_{ik}\nonumber\\
	&&-2uW_{ikjp}\R_{ij}\R_{kp}+\frac{4u}{n-2}\R_{ip}\R_{ij}\R_{pj}+\frac{2Ru}{n(n-1)}|\Rc|^2\nonumber\\
	&=&u|\nabla\mathring{Ric}|^{2}+\frac{m(n-2)u}{m+n-2}|C|^2+u\langle\nabla^2R,\Rc\rangle+\frac{m+2n-2}{n-1}\Rc(\nabla u,\nabla R)\nonumber\\
	&&+\frac{m-1}{2}\langle\nabla|\Rc|^2,\nabla u\rangle+\frac{2Ru}{n-1}|\Rc|^2+\frac{2nu}{n-2}\R_{ip}\R_{ij}\R_{pj}\nonumber\\
	&&-2uW_{ikjp}\R_{ij}\R_{kp}-\frac{m^2(n-2)}{m+n-2}C_{ijk}W_{ijkl}\nabla_lu\nonumber,
\end{eqnarray} which finishes the proof of the theorem. 
\end{proof}

\subsection{Proof of Corollary \ref{corWeylZero}}

\begin{proof}
Initially, taking into account that $M^n$ has constant scalar curvature and zero radial Weyl curvature, we may use Theorem \ref{th3} to infer
\begin{eqnarray}\label{corWeylZeroeq1}
\frac{1}{2}\div \big(u\nabla|\mathring{Ric}|^{2}\big)&=&u|\nabla\mathring{Ric}|^{2}+\frac{m(n-2)u}{m+n-2}|C|^2+\frac{m-1}{2}\langle\nabla|\Rc|^2,\nabla u\rangle\nonumber\\
&&+\frac{2Ru}{n-1}|\Rc|^2+\frac{2nu}{n-2}{\rm tr}(\Rc^3)-2uW_{ikjp}\R_{ij}\R_{kp}.
\end{eqnarray} 

On the other hand, we use  Eqs. (\ref{traceless Ricci}) and (\ref{cottonwyel}) and that $M^n$ has zero radial Weyl curvature in order to obtain
\begin{eqnarray*}
uW_{ijkl}\R_{ik}\R_{jl}&=&mW_{ijkl}\R_{ik}(\nabla_j\nabla_lu-\frac{\Delta u}{n}g_{jl})\nonumber\\
&=&m\nabla_j(W_{ijkl}\R_{ik}\nabla_lu)-m\nabla_jW_{ijkl}\nabla_lu\R_{ik}-mW_{ijkl}\nabla_j\R_{ik}\nabla_lu\nonumber\\
&=&\frac{m(n-3)}{n-2}C_{kli}\nabla_lu\R_{ik}\nonumber\\
&=&\frac{m(n-3)}{2(n-2)}(C_{kli}R_{ik}\nabla_lu+C_{lki}R_{il}\nabla_ku)\nonumber\\
&=&\frac{m(n-3)}{2(n-2)}C_{kli}(R_{ik}\nabla_lu-R_{il}\nabla_ku),
\end{eqnarray*} where we used that the Cotton tensor is trace-free. Combining this expression with Lemma \ref{CWT} and (\ref{T-tensor}) we deduce
\begin{eqnarray}\label{corWeylZeroeq2}
uW_{ijkl}\R_{ik}\R_{jl}&=&\frac{m(n-3)}{2(m+n-2)}C_{kli}T_{kli}\nonumber\\
&=&\frac{m(n-3)u}{2(m+n-2)}|C|^2.
\end{eqnarray}

Proceeding, since $R$ is constant, we can invoke Proposition 3.3 of \cite{Petersen and Chenxu} (see also Lemma 3.2 of \cite{CaseShuWey}) to deduce
\begin{eqnarray*}
|\Rc|^2=-\frac{m+n-1}{n(m-1)}(R-n\lambda)\left(R-\frac{n(n-1)}{m+n-1}\lambda\right).
\end{eqnarray*} In particular, $|\Rc|^2$ is also constant. Thus, substituting (\ref{corWeylZeroeq2}) into (\ref{corWeylZeroeq1}) and using that $|\Rc|^2$ is constant we get
\begin{eqnarray}\label{corWeylZeroeq3}
0&=&u|\nabla\mathring{Ric}|^{2}+\frac{m(n-2)u}{m+n-2}|C|^2+\frac{2Ru}{n-1}|\Rc|^2\nonumber\\
 &&+\frac{2nu}{n-2}{\rm tr}(\Rc^3)-\frac{m(n-3)u}{m+n-2}|C|^2\nonumber\\
&=&u|\nabla\mathring{Ric}|^{2}+\frac{mu}{m+n-2}|C|^2+\frac{2Ru}{n-1}|\Rc|^2+\frac{2nu}{n-2}{\rm tr}(\Rc^3).
\end{eqnarray} From the Okumura's Lemma (see Lemma 2.1 \cite{Okumura}) we obtain
$${\rm trace}(\Rc^3)\geq-\frac{n-2}{\sqrt{n(n-1)}}|\Rc|^3.$$
Replacing this into (\ref{corWeylZeroeq3}) we get
\begin{eqnarray}\label{corWeylZeroeq4}
0&\geq&u|\nabla\mathring{Ric}|^{2}+\frac{mu}{m+n-2}|C|^2+\frac{2Ru}{n-1}|\Rc|^2-\frac{2nu}{\sqrt{n(n-1)}}|\Rc|^3\nonumber\\
 &\geq &u|\nabla\mathring{Ric}|^{2}+\frac{mu}{m+n-2}|C|^2+\frac{2n}{\sqrt{n(n-1)}}\left(\frac{R}{\sqrt{n(n-1)}}-|\Rc|\right)u|\Rc|^2.
\end{eqnarray} 
Finally, as $R$ is positive and using that $|\Rc|^2<\frac{R^2}{n(n-1)}$  in (\ref{corWeylZeroeq4}), we conclude that $\Rc=0$ and therefore, $(M^n,g)$ is an Einstein manifold. Now, it suffices to apply Proposition 2.4 of \cite{Petersen and Chenxu} to conclude that $(M^{n},\,g)$ is isometric to the hemisphere $\Bbb{S}^n_+.$ So, the proof is finished. 
\end{proof}

\begin{acknowledgement}
The authors want to thank the referee for careful rea\-ding and valuable suggestions. Moreover, they want to thank E. Ribeiro Jr. for helpful conversations about this subject. 
\end{acknowledgement}

\end{document}